\documentclass[11pt]{article}
\usepackage{latexsym,amssymb,float,amsmath,amsthm,enumerate,cite,geometry,tikz}
\newtheorem{theorem}{Theorem}

\newtheorem{lemma}[theorem]{Lemma}
\newtheorem{corollary}[theorem]{Corollary}

\usepackage{lineno}
\allowdisplaybreaks

\begin{document}

%\linenumbers
%\onehalfspace
\title{An Optimization Approach to Degree Deviation and Spectral Radius}
\author{Dieter Rautenbach \and Florian Werner}
\date{}
\maketitle
\vspace{-1cm}
\begin{center}
Institute of Optimization and Operations Research, Ulm University, Ulm, Germany\\
\texttt{$\{$dieter.rautenbach,florian.werner$\}$@uni-ulm.de}
\end{center}
\begin{abstract}
For a finite, simple, and undirected graph $G$ 
with $n$ vertices and average degree $d$,
Nikiforov introduced the degree deviation of $G$ as
$s=\sum_{u\in V(G)}\left|d_G(u)-d\right|$.
Provided that $G$ has 
largest eigenvalue $\lambda$,
minimum degree at least $\delta$, and 
maximum degree at most $\Delta$,
where $0\leq\delta<d<\Delta<n$,
we show
$$s\leq \frac{2n(\Delta-d)(d-\delta)}{\Delta-\delta}
\,\,\,\,\,\,\,\,\,\,\,\,\,\,\,\,\mbox{and}\,\,\,\,\,\,\,\,\,\,\,\,\,\,\,\,
\lambda
\geq
\begin{cases}
\frac{d^2n}{\sqrt{d^2n^2-s^2}} & 
\mbox{, if } s\leq \frac{dn}{\sqrt{2}},\\[3mm]
\frac{2s}{n} & 
\mbox{, if } s> \frac{dn}{\sqrt{2}}.
\end{cases}$$
Our results are based on a smoothing technique
relating the degree deviation 
and the largest eigenvalue 
to low-dimensional non-linear optimization problems.\\[3mm]
\end{abstract}

\section{Introduction}\label{sec1}

We consider finite, simple, and undirected graphs and 
use standard notation and terminology.
Throughout the introduction let $G$ be a non-empty graph 
with $n$ vertices, $m$ edges, and average degree $d=\frac{2m}{n}$.

Nikiforov~\cite{ni} defined the {\it degree deviation} $s(G)$ of $G$ as
\begin{eqnarray}\label{e-1}
s(G)=\sum\limits_{u\in V(G)}\left|d_G(u)-d\right|,
\end{eqnarray}
where $V(G)$ is the vertex set of $G$
and $d_G(u)$ is the degree of the vertex $u$ in $G$.

In~\cite{latiha} Lawrence, Tizzard, and Haviland 
already considered $s(G)/n$ as the {\it discrepancy} of $G$ and
Haviland~\cite{ha} showed 
\begin{eqnarray}\label{e0}
s(G) & \leq & \psi\left(2n-1-\sqrt{4n\psi+1}\right),
\end{eqnarray}
where $\psi=\min\left\{d,n-1-d\right\}$.
She observed that
if $m={q\choose 2}\leq \frac{1}{2}{n\choose 2}$ for some integer $q$,
then $K_q\cup \overline{K}_{n-q}$ satisfies (\ref{e0}) with equality and 
if $m={t\choose t}+t(n-t)\geq \frac{1}{2}{n\choose 2}$ for some integer $t$,
then $K_t+\overline{K}_{n-t}=\overline{\overline{K}_t\cup K_{n-t}}$ 
satisfies (\ref{e0}) with equality,
which implies that (\ref{e0}) is best possible up to terms of smaller order
for all possible choices of $n$ and $m$.
Provided that $n\geq 2$
and that the graph $G$ is connected, 
has minimum degree $\delta$, and maximum degree $\Delta$,
Ali at al.~\cite{almimami} showed
\begin{eqnarray}\label{e1}
s(G) & \leq & \frac{2m}{n}\sqrt{\frac{2m(n(\Delta+\delta)-2m)-n^2\Delta\delta}{\delta\Delta}}
=dn\sqrt{\frac{(\Delta-d)(d-\delta)}{\delta\Delta}}.
\end{eqnarray}
They claim in~\cite{almimami} that (\ref{e1}) is satisfied with equality if and only if 
$G$ is a regular or semiregular $(\Delta,\delta)$-bipartite graph,
which is not quite true.
In Section \ref{sec2}, we reconsider the proof of (\ref{e1}) 
and determine the extremal graphs.
As it turns out, unlike (\ref{e0}),
the bound (\ref{e1}) is only best possible 
for $d=\frac{2\delta\Delta}{\delta+\Delta}$.

In Section \ref{sec4} we provide the following 
upper bound on the degree deviation.

\begin{theorem}\label{theorem1}
If $G$ is a graph with 
$n$ vertices, 
average degree $d$, 
minimum degree at least $\delta$, and 
maximum degree at most $\Delta$,
where $0\leq\delta<d<\Delta<n$,
then
$$s(G)\leq \frac{2n(\Delta-d)(d-\delta)}{\Delta-\delta}.$$
\end{theorem}
Theorem \ref{theorem1} is best possible up to terms of smaller order 
for a wide range of values and outperforms (\ref{e1});
cf.~Section \ref{sec4} for details.

For the {\it spectral radius} $\lambda(G)$ of $G$,
which is the largest eigenvalue of the adjacency matrix of $G$,
Nikiforov~\cite{ni} showed that
\begin{eqnarray}\label{eni}
\frac{s(G)^2}{2n^2\sqrt{2m}}\leq \lambda(G)-d\leq \sqrt{s(G)}.
\end{eqnarray}
He conjectured that both inequalities in (\ref{eni})
can be improved by a factor of $\sqrt{2}$ for sufficiently large $n$ and $m$,
that is, that 
\begin{eqnarray}\label{eni2}
\frac{s(G)^2}{n^2\sqrt{2dn}}=\frac{s(G)^2}{2n^2\sqrt{m}}\leq \lambda(G)-d\leq \sqrt{\frac{s(G)}{2}}.
\end{eqnarray}
Zhang~\cite{zh} showed 
$\lambda(G)-d\leq \sqrt{\frac{9s(G)}{10}}$,
and we~\cite{rawe} recently showed 
$\lambda(G)-d\leq \sqrt{\frac{2s(G)}{3}}$.
In Section \ref{sec5} we show the following lower bound on the spectral radius,
where the quantity $\tilde{\lambda}(G)$ is defined in Section \ref{sec3} 
just before Lemma \ref{lemma2}.

\begin{theorem}\label{theorem3}
If $G$ is a graph with 
$n$ vertices, 
average degree $d$, and 
degree deviation $s$ with $s>0$, then
$$\lambda(G)\geq\tilde{\lambda}(G)\geq
\begin{cases}
\frac{d^2n}{\sqrt{d^2n^2-s^2}} & 
\mbox{, if } s\leq \frac{dn}{\sqrt{2}},\mbox{ and}\\[3mm]
\frac{2s}{n} & 
\mbox{, if } s> \frac{dn}{\sqrt{2}}.
\end{cases}
$$
\end{theorem}
Theorem \ref{theorem3} implies 
the conjectured lower bound on $\lambda(G)-d$ in (\ref{eni2})
for $d\leq n/2$., cf.~Corollary \ref{corollary1} below.\footnote{In an appendix 
we complete the proof for $d>n/2$.}
Our results in Sections \ref{sec4} and \ref{sec5} rely 
on a smoothing technique introduced in Section \ref{sec3}.
This technique leads to low dimensional non-linear optimization problems
related to the degree deviation and the spectral radius.

\section{Smoothing}\label{sec3}

Throughout this section, 
let $n$, $d$, $\delta$, and $\Delta$ 
with $0\leq\delta<d<\Delta<n$ 
be such that $n$, $\delta$, $\Delta$, and $m=\frac{dn}{2}$
are integers.
Furthermore, let $G$ be a graph with 
vertex set $V$,
$n$ vertices, 
average degree $d$,
minimum degree at least $\delta$, and 
maximum degree at most $\Delta$.
Note that $\delta$ and $\Delta$ are degree bounds within this section
and that $G$ is not required to contain vertices of these degrees.

Let\\[-10mm]
\begin{eqnarray*}
V_+&=&\{ u\in V:d_G(u)>d\},\\
V_-&=&\{ u\in V:d_G(u)\leq d\},\\
n_+&=&|V_+|,\\
n_-&=&|V_-|,\\
m_+&=&m(G[V_+]),\\ 
m_-&=&m(G[V_-])\mbox{, and}\\
m_{\pm}&=&m-(m_++m_-).
\end{eqnarray*}
Let $K$ be the complete graph with $V(K)=V$ and let
$$
w:E(K)\to \mathbb{R}:
uv\mapsto 
\begin{cases}
w_+=\frac{m_+}{{n_+\choose 2}} & \mbox{, if $u,v\in V_+$,}\\
w_-=\frac{m_-}{{n_-\choose 2}} & \mbox{, if $u,v\in V_-$, and}\\
w_{\pm}=\frac{m_{\pm}}{n_+n_-} & \mbox{, if $u\in V_+$ and $v\in V_-$.}
\end{cases}
$$
We recall some usual notions for edge-weighted graphs.

Let $n((K,w))=n$ and $m((K,w))=\sum\limits_{uv\in {V\choose 2}}w(uv)$.

For a vertex $u$ from $V$, let the {degree of $u$ in $(K,w)$} be
$d_{(K,w)}(u)=\sum\limits_{v\in V\setminus \{ u\}}w(uv)$.

Let 
$$
s((K,w))=\sum_{u\in V}\left|d_{(K,w)}(u)-\frac{2m_{(K,w)}}{n_{(K,w)}}\right|.
$$
Considering $(K,w)$ as the smoothed version of $G$,
the following lemma shows that relevant quantities are invariant under smoothing.

\begin{lemma}\label{lemma1}
If $G$ and $(K,w)$ are as above, then the following statements hold.
\begin{enumerate}[(i)]
\item $m((K,u))=m$.
\item $d<d_{(K,w)}(u)\leq \Delta$ for every $u$ in $V_+$.
\item $\delta\leq d_{(K,w)}(u)\leq d$ for every $u$ in $V_-$.
\item $\mbox{}$\\[-12mm]
\begin{eqnarray*}
s(G) &=& s((K,w))\\
&=&n_+\Big(w_+(n_+-1)-d\Big)+n_-\Big(d-w_-(n_--1)\Big)\\
&=&2(m_+-m_-)-d(n_+-n_-).
\end{eqnarray*}
\end{enumerate}
\end{lemma}
\begin{proof} Since (i) is straightforward and (iii) is symmetric to (ii), 
we give details only for (ii) and (iv).
We have
\begin{eqnarray}
\sum\limits_{u\in V_+}d_{(K,w)}(u)
&=& \sum\limits_{u\in V_+}\sum\limits_{v\in V\setminus \{ u\}}w(uv)\nonumber\\
&=& 
\sum\limits_{u\in V_+}\sum\limits_{v\in V_+\setminus \{ u\}}w(uv)
+
\sum\limits_{u\in V_+}\sum\limits_{v\in V_-}w(uv)\nonumber\\
&=& 
\sum\limits_{u\in V_+}\sum\limits_{v\in V_+\setminus \{ u\}}\frac{m_+}{{n_+\choose 2}}
+
\sum\limits_{u\in V_+}\sum\limits_{v\in V_-}\frac{m_{\pm}}{n_+n_-}\nonumber\\
&=& 2m_++m_{\pm}\nonumber\\
& = & \sum\limits_{u\in V_+}d_G(u).\label{e2}
\end{eqnarray}
By a symmetric argument, we also obtain
\begin{eqnarray}
\sum\limits_{u\in V_-}d_{(K,w)}(u)=\sum\limits_{u\in V_-}d_G(u).\label{e3}
\end{eqnarray}
If $n_+=0$, then (ii) is void.
Hence, for the proof of (ii), we may assume that $n_+>0$.
By the definition of $(K,w)$, 
the degree function $d_{(K,w)}(\cdot)$ is constant on $V_+$.
Since $d<d_G(u)\leq \Delta$ for every $u\in V_+$, 
we obtain using (\ref{e2}) that
$dn_+<d_{(K,w)}(u)n_+\leq \Delta n_+$
for every $u\in V_+$,
which implies (ii).

Since $\frac{2m_{(K,w)}}{n_{(K,w)}}=\frac{2m}{n}=d$,
we obtain using (ii) and (iii) that
\begin{eqnarray*}
s((K,w)) 
&=& 
\sum_{u\in V_+}\left(d_{(K,w)}(u)-d)\right)
+\sum_{u\in V_-}\left(d-d_{(K,w)}(u))\right)\\
&=& 
d(n_--n_+)+\sum_{u\in V_+}d_{(K,w)}(u)-\sum_{u\in V_-}d_{(K,w)}(u)\\
&\stackrel{(\ref{e2}),(\ref{e3})}{=}& 
d(n_--n_+)+\sum\limits_{u\in V_+}d_G(u)-\sum\limits_{u\in V_-}d_G(u)\\
&=& 
\sum_{u\in V_+}\left(d_G(u)-d)\right)
+\sum_{u\in V_-}\left(d-d_G(u))\right)%\\& =& 
=s(G).
\end{eqnarray*}
Since 
$d_{(K,w)}(u)=w_+(n_+-1)+w_{\pm}n_->d$ for $u\in V_+$
and 
$d_{(K,w)}(u)=w_-(n_--1)+w_{\pm}n_+\leq d$ for $u\in V_-$,
it follows that (iv) holds.
\end{proof}
Let $A=(a_{u,v})_{u,v\in V}$ be the adjacency matrix of $G$
and let $\lambda=\lambda(G)$ be the largest eigenvalue of $A$.
The natural choice for the (smoothed) adjacency matrix of $(K,w)$
is the matrix $\tilde{A}=(\tilde{a}_{u,v})_{u,v\in V}$ with
$$
\tilde{a}_{u,v}=
\begin{cases}
w_+ & \mbox{, if $u,v\in V_+$ with $u\not= v$,}\\
w_- & \mbox{, if $u,v\in V_-$ with $u\not= v$,}\\
w_{\pm} & \mbox{, if $u\in V_+$ and $v\in V_-$, and}\\
0 & \mbox{, if $u=v$.}
\end{cases}$$

Let $\tilde{\lambda}=\tilde{\lambda}(G)$ be the largest eigenvalue of $\tilde{A}$.

\begin{lemma}\label{lemma2}
For $\lambda$ and $\tilde{\lambda}$ as above, we have
$$\lambda \geq \tilde{\lambda}\geq d$$
and
$$\tilde{\lambda}=\frac{1}{2}\left(\Big(n_+-1\Big)w_++\Big(n_--1\Big)w_-+
\sqrt{\Big((n_+-1)w_+-(n_--1)w_-\Big)^2+4n_+n_-w^2_{\pm}}
\right).$$
\end{lemma}
\begin{proof}
It is well-known that the eigenvectors $x$ of 
the symmetric non-negative matrix $\tilde{A}$ 
for its largest eigenvalue $\tilde{\lambda}$ 
are exactly the maximizers of the Rayleigh quotient $\frac{x^T\tilde{A}x}{x^Tx}$.
Choosing $x\in \mathbb{R}^V$ as the all-one vector,
it follows that $\tilde{\lambda}
\geq \frac{x^T\tilde{A}x}{x^Tx}
=\frac{2m}{n}=d$.
Now, we choose an eigenvector $x=(x_u)_{u\in V}$ 
of $\tilde{A}$ for the eigenvalue $\tilde{\lambda}$ 
in such a way that $x^Tx=1$ and 
$||x||_1=\sum_{u\in V}|x_u|$ is maximized.
Since $\tilde{A}$ is non-negative,
we may assume that $x$ is non-negative.

Our first goal is to show that $x_u$ is constant on $V_+$ and $V_-$.
Suppose that $x_u<x_v$ for $u,v\in V_+$.
The structure of $\tilde{A}$ implies that 
$x^T\tilde{A}x=\alpha x_ux_v+\beta (x_u+x_v)+\gamma$,
where $\alpha$, $\beta$, and $\gamma$ are non-negative
and depend only on $\tilde{A}$ and on $(x_w)_{w\in V\setminus\{ u,v\}}$.
Let 
$y=(y_u)_{u\in V}$ be such that
$y_u=y_v=\sqrt{\frac{x_u^2+x_v^2}{2}}$
and $y_w=x_w$ for every $w\in V\setminus\{ u,v\}$.
Since $x_u<x_v$,
we obtain $x_ux_v<\frac{x_u^2+x_v^2}{2}=y_uy_v$
and $x_u+x_v<\sqrt{2(x_u^2+x_v^2)}=y_u+y_v$.
Now, it follows that $y^Ty=1$, $||y||_1>||x||_1$, and 
$$x^T\tilde{A}x
=\alpha x_ux_v+\beta (x_u+x_v)+\gamma
\leq \alpha y_uy_v+\beta (y_u+y_v)+\gamma
=y^T\tilde{A}y.$$
It follows that $y$ maximizes the Rayleigh quotient and,
hence, is an eigenvector of $\tilde{A}$ for $\tilde{\lambda}$,
which implies a contradiction to the choice of $x$.
Hence, the value $x_u$ is constant on $V_+$ and, 
by symmetry, also on $V_-$.

Let $x_u=x_+$ for $u\in V_+$ and $x_u=x_-$ for $u\in V_-$.
Since $x$ is a normalized eigenvector of $\tilde{A}$ for $\tilde{\lambda}$,
we obtain
\begin{eqnarray*}
1 &=& n_+x_+^2+n_-x_-^2,\\
\tilde{\lambda} x_+ &=& (n_+-1)w_+x_++n_-w_{\pm}x_-,\mbox{ and }\\
\tilde{\lambda} x_- &=& (n_--1)w_-x_-+n_+w_{\pm}x_+.
\end{eqnarray*}
By symmetry, we may assume $x_+>0$.

Setting $y=\frac{x_-}{x_+}$, the second and third equation yield
$$
y=\frac{1}{2n_-w_{\pm}}
\left(-\Big((n_+-1)w_+-(n_--1)w_-\Big)+
\sqrt{\Big((n_+-1)w_+-(n_--1)w_-\Big)^2+4n_+n_-w^2_{\pm}}
\right),
$$
which yields the stated value for $\tilde{\lambda}$.
Since $x_u$ is constant on $V_+$ and $V_-$,
the definition of $\tilde{A}$ implies
$$\tilde{\lambda}
=\frac{x^T\tilde{A}x}{x^Tx}
=\frac{x^TAx}{x^Tx}
\leq \lambda,$$
which completes the proof.
\end{proof}
Note that $\tilde{A}$ and $A$, and thus also $\tilde{\lambda}$ and $\lambda$, coincide
provided that $w_+,w_-,w_{\pm}\in \{ 0,1\}$, 
which is the case for the graphs
$K_q\cup \overline{K}_{n-q}$
and 
$K_t+\overline{K}_{n-t}$.

\section{Upper bound on $s(G)$ in terms of $n$, $d$, $\delta$, and $\Delta$}\label{sec4}

Throughout this section, 
let $n$, $d$, $\delta$, and $\Delta$ 
with $0\leq\delta<d<\Delta<n$ 
be such that $n$, $\delta$, $\Delta$, and $m=\frac{dn}{2}$
are integers.
Let the graph $G$ with 
$n$ vertices, 
average degree $d$,
minimum degree at least $\delta$, and 
maximum degree at most $\Delta$
maximize the deviation $s(G)$.
If $s(G)=0$, then $G$ is $d$-regular and 
there is some edge $uv$ and a vertex $w$ such that $u$ is not adjacent to $w$.
Now, the graph $G'=G-uv+uw$ satisfies $s(G')>0$
and contradicts the choice of $G$.
Hence, we obtain that $s(G)>0$,
which implies $n\geq 3$.

The results of Section \ref{sec3} imply 
$$s(G)\leq {\rm OPT}(P_I)$$ 
for the following non-linear optimization problem $(P_I)$
depending on the parameters $n$, $m$, $\delta$, and $\Delta$:

$$
(P_I)\hspace{2cm}\begin{array}{lrlll}
\mbox{maximize} & n_+\Big(w_+(n_+-1)-d\Big)+n_-\Big(d-w_-(n_--1)\Big) & & &\\[4mm]
\mbox{subject to} & n_++n_- & = & n &\\
						& {n_+\choose 2}w_++n_+n_-w_{\pm}+{n_-\choose 2}w_- & = & m & \\
						& (n_+-1)w_++n_-w_{\pm} & \in & [d,\Delta] & \\
						& (n_--1)w_-+n_+w_{\pm} & \in & [\delta,d] & \\
						& n_+,n_- & \in & \mathbb{N}_0=\{ 0,1,2,\ldots\} &\\
						& w_+,w_-,w_{\pm} & \in & [0,1] &
\end{array}
$$
Let 
\begin{eqnarray*}
f(n_+,n_-,w_+,w_-)&=&
n_+\Big(w_+(n_+-1)-d\Big)+n_-\Big(d-w_-(n_--1)\Big),\\
d_+&=&(n_+-1)w_++n_-w_{\pm},\mbox{ and}\\ 
d_-&=&(n_--1)w_-+n_+w_{\pm}.
\end{eqnarray*}
For fixed $n_+\in \{ 0,1,\ldots,n\}$, we have $n_-=n-n_+$ and 
$(P_I)$ reduces to a linear program in terms of the variables 
$w_+$, $w_-$, and $w_{\pm}$.
In particular, for given $(n,m,\delta,\Delta)$, 
the value ${\rm OPT}(P_I)$ can be determined by solving $n+1$ linear programs.  

The constraints of $(P_I)$ imply $dn_++dn_-=dn=2m=d_+n_++d_-n_-$.
If $d_+=d$, then this implies $d_-=d$ and ${\rm OPT}(P_I)=0$, 
contradicting $s(G)>0$.
Hence, it follows that $d_+>d$.
Similarly, it follows that $d_-<d$.
Since $dn=d_+n_++d_-n_-$, 
it follows that every feasible solution of $(P_I)$ satisfies $n_+,n_->0$.
Let $(P)$ denote the relaxation of $(P_I)$,
where the condition ``$n_+,n_-\in\mathbb{N}_0$'',
which could be strenthened to ``$n_+,n_-\in\mathbb{N}$''
without changing the set of feasible solutions,
is relaxed to ``$n_+,n_-\geq 1$''.

We obtain
$$s(G)\leq {\rm OPT}(P_I)\leq {\rm OPT}(P)=:s(n,m,\delta,\Delta).$$
Since the problem $(P)$ consists in maximizing the continuous function $f$ 
on a compact domain, there are optimal solutions.
As noted above, we have $s(n,m,\delta,\Delta)>0$.
 
\begin{lemma}\label{lemma3}
Let $x=(n_+,n_-,w_+,w_-,w_{\pm})$ be an optimal solution for $(P)$.
\begin{enumerate}[(i)]
\item $d_+>d$ and $d_-<d$. 
\item If $n_+=1$, then 
$f(n_+,n_-,w_+,w_-)\leq \min\big\{2(\Delta-d),2(n-1)(d-\delta)\big\}$.
\item If $n_-=1$, then 
$f(n_+,n_-,w_+,w_-)\leq \min\big\{2(n-1)(\Delta-d),2(d-\delta)\big\}$.
\item If $n_+,n_->1$, then $d_+=\Delta$ or $d_-=\delta$.
\item $f(n_+,n_-,w_+,w_-)\leq \frac{2n(\Delta-d)(d-\delta)}{\Delta-\delta}$.
\item If $\delta=0$, $d\leq n-3$, and $2\Delta\leq dn<\Delta\left(\Delta+1\right)$, then
$$f(n_+,n_-,w_+,w_-)\leq \max\left\{
\left(2\Delta+1-\sqrt{(2\Delta+1)^2-4dn}\right)(\Delta-d),
d\left(2n-1-\sqrt{4dn+1}\right)
\right\}.$$
\end{enumerate}
\end{lemma}
\begin{proof}
(i) We can repeat the above argument:
The constraints of $(P)$ imply $dn_++dn_-=d_+n_++d_-n_-$.
If $d_+=d$, then this implies $d_-=d$ and ${\rm OPT}(P)=0$, which is a contradiction.
Hence, it follows that $d_+>d$.
Similarly, it follows that $d_-<d$.

\medskip

\noindent (ii) Let $n_+=1$. 
The two equality constraints in $(P)$ imply
$n_-=n-1$ and $w_{\pm}=\frac{dn-(n-2)(n-1)w_-}{2(n-1)}$.
Substituting these expressions into $d_+$ and $d_-$
and exploiting the constraints $d_+\leq \Delta$ and $d_-\geq \delta$,
we obtain
$w_-\geq \max\left\{\frac{d}{n-1}-\frac{2(\Delta-d)}{(n-2)(n-1)},\frac{d}{n-1}-\frac{2(d-\delta)}{(n-2)}\right\}$.
Now,
\begin{eqnarray*}
f(1,n-1,w_+,w_-,w_{\pm}) &=& (n-2)d-(n-2)(n-1)w_-
\leq \min\big\{2(\Delta-d),2(n-1)(d-\delta)\big\}.
\end{eqnarray*}

\medskip

\noindent (iii) Let $n_-=1$. 
The two equality constraints in $(P)$ imply
$n_+=n-1$ and $w_{\pm}=\frac{dn-(n-2)(n-1)w_+}{2(n-1)}$.
Substituting these expressions into $d_+$ and $d_-$
and exploiting the constraints $d_+\leq \Delta$ and $d_-\geq \delta$,
we obtain
$w_+\leq \min\left\{\frac{d}{n-1}+\frac{2(\Delta-d)}{(n-2)},\frac{d}{n-1}+\frac{2(d-\delta)}{(n-2)(n-1)}\right\}$.
Now,
\begin{eqnarray*}
f(n-1,1,w_+,w_-,w_{\pm}) &=& (n-2)(n-1)w_+-(n-2)d
\leq \min\big\{2(n-1)(\Delta-d),2(d-\delta)\big\}.
\end{eqnarray*}

\medskip

\noindent (iv) Suppose, for a contradiction, that 
 $n_+,n_->1$, $d_+<\Delta$, and $d_->\delta$.
This implies $w_++w_{\pm}<2$ and $w_-+w_{\pm}>0$.
By (i), 
the quantities ${n_+\choose 2}$, $n_+n_-$, and ${n_-\choose 2}$ are positive.
If $w_->0$, then reducing $w_-$ by some sufficiently small amount $\epsilon_1>0$
and increasing the smaller of the two values $w_+$ and $w_{\pm}$ 
by some small amount $\epsilon_2>0$
in such a way that the equality
${n_+\choose 2}w_++n_+n_-w_{\pm}+{n_-\choose 2}w_- = m$ is maintained,
yields a feasible solution increasing the objective function value $f(n_+,n_-,w_+,w_-)$,
which contradicts the optimality of $x$.
Hence, we obtain that $w_-=0$.
Similarly, it follows that $w_+=1$.
Since $d_+<\Delta$ and $d_->\delta$, it follows that $0<w_{\pm}<1$.
Note that 
\begin{eqnarray*}
f(n_++\epsilon,n_--\epsilon,1,0)-f(n_+,n_-,1,0)=
\epsilon(\epsilon+2(n_1-d)-1).
\end{eqnarray*}
If $2(n_1-d)-1\geq 0$, then let $\epsilon$ be positive,
and if $2(n_1-d)-1<0$, then let $\epsilon$ be negative. 
Choosing $|\epsilon|$ sufficiently small, 
adding $\epsilon$ to $n_+$,
subtracting $\epsilon$ from $n_-$, and 
adapting $w_{\pm}$ 
in such a way that the equality
${n_+\choose 2}w_++n_+n_-w_{\pm}+{n_-\choose 2}w_- = m$ is maintained,
yields a feasible solution increasing the objective function value,
which contradicts the optimality of $x$.
This completes the proof of (iv).

\medskip

\noindent (v) If $n_+=1$, then (ii) implies 
\begin{eqnarray*}
f(n_+,n_-,w_+,w_-)
&\leq &\min\big\{2(\Delta-d),2(n-1)(d-\delta)\big\}\\
&\leq &2(\Delta-d)\frac{(d-\delta)}{(\Delta-\delta)}+2(n-1)(d-\delta)
\left(1-\frac{(d-\delta)}{(\Delta-\delta)}\right)\\
&=& \frac{2n(\Delta-d)(d-\delta)}{\Delta-\delta}.
\end{eqnarray*}
If $n_-=1$, then (iii) implies
\begin{eqnarray*}
f(n_+,n_-,w_+,w_-)
&\leq & \min\big\{2(n-1)(\Delta-d),2(d-\delta)\big\}\\
&\leq &2(n-1)(\Delta-d)\frac{(d-\delta)}{(\Delta-\delta)}+2(d-\delta)
\left(1-\frac{(d-\delta)}{(\Delta-\delta)}\right)\\
&=& \frac{2n(\Delta-d)(d-\delta)}{\Delta-\delta}.
\end{eqnarray*}
Now, we may assume that $n_+,n_->1$
and (iv) implies that $d_+=\Delta$ or $d_-=\delta$.

First, we suppose that $d_+=\Delta$. 
Using the three equations
$n_++n_- =n$,
$(n_+-1)w_++n_-w_{\pm}=\Delta$, and
${n_+\choose 2}w_++n_+n_-w_{\pm}+{n_-\choose 2}w_-=m=\frac{dn}{2}$,
we obtain
\begin{eqnarray}\label{elim1}
n_- = n-n_+,\,\,\,\,\,\,\,\,
w_{\pm}= \frac{\Delta-(n_+-1)w_+}{n-n_+},\,\,\mbox{ and}\,\,\,\,\,\,\,\,\,\,
w_-= \frac{dn-(2\Delta+w_+)n_++w_+n_+^2}{(n-n_+)(n-n_+-1)}.
\end{eqnarray}
Substituting these expressions, 
it follows that
\begin{eqnarray}\label{elim2}
f:=f(n_+,n_-,w_+,w_-)=2(\Delta-d)n_+
\end{eqnarray}
and
$d_-=\frac{dn-\Delta n_+}{n-n_+}$.
The condition $d_-\geq \delta$ implies $n_+\leq \frac{(d-\delta)n}{\Delta-\delta}$
and, hence, in this case
$$f\leq \frac{2n(\Delta-d)(d-\delta)}{\Delta-\delta}$$
as stated.

Next, we suppose that $d_-=\delta$.
Using the three equations
$n_++n_- =n$,
$(n_--1)w_-+n_-w_{\pm}=\delta$, and
${n_+\choose 2}w_++n_+n_-w_{\pm}+{n_-\choose 2}w_-=m=\frac{dn}{2}$,
we obtain
\begin{eqnarray*}
n_+ = n-n_-,\,\,\,\,\,\,\,\,
w_{\pm}= \frac{\delta-(n_--1)w_-}{n-n_-},\,\,\mbox{ and}\,\,\,\,\,\,\,\,\,\,
w_+= \frac{dn-(2\delta+w_-)n_-+w_-n_-^2}{(n-n_-)(n-n_--1)}.
\end{eqnarray*}
Substituting these expressions, 
we obtain 
\begin{eqnarray}\label{elim2b}
f=2(d-\delta)n_-
\end{eqnarray}
and
$d_+=\frac{dn-\delta n_-}{n-n_-}$.
The condition $d_+\leq \Delta$ implies $n_-\leq \frac{(\Delta-d)n}{\Delta-\delta}$
and, hence, also in this case
$$f\leq \frac{2n(\Delta-d)(d-\delta)}{\Delta-\delta}$$
as stated.
This completes the proof of (v).

\medskip

\noindent (vi) If $n_+=1$, then $\Delta\leq m$ implies
$2\Delta+1-\sqrt{(2\Delta+1)^2-4dn}\geq 2$ and (ii) implies 
$f\leq 2(\Delta-d)
\leq \left(2\Delta+1-\sqrt{(2\Delta+1)^2-4dn}\right)(\Delta-d)$.
If $n_-=1$, then $d\leq n-3$ implies 
$2n-1-\sqrt{4dn+1}\geq 2$ and (iii) implies 
$f\leq 2d\leq d\left(2n-1-\sqrt{4dn+1}\right)$.
Now, we may assume that $n_+,n_->1$
and (iv) implies that $d_+=\Delta$ or $d_-=\delta$.

First, we suppose that $d_+=\Delta$. 
As in the proof of (iii), 
we obtain (\ref{elim1}) and (\ref{elim2}).
If $w_+=0$, then $w_-\geq 0$ and (\ref{elim1}) imply $n_+\leq \frac{dn}{2\Delta}$.
Together with (\ref{elim2}) this implies $f=2(\Delta-d)n_+\leq \frac{dn}{\Delta}(\Delta-d)$.
Since $2\Delta\leq dn$, we have 
$\frac{dn}{\Delta}\leq 2\Delta+1-\sqrt{(2\Delta+1)^2-4dn}$
and we obtain $f\leq \left(2\Delta+1-\sqrt{(2\Delta+1)^2-4dn}\right)(\Delta-d)$.
Now, we may assume that $w_+>0$.
Since $w_-\geq 0$, it follows using (\ref{elim1}) that 
\begin{eqnarray}
\label{elim3}
\mbox{either }
n_+ & \leq & \frac{1}{2w_+}\left(2\Delta+w_+-\sqrt{(2\Delta+w_+)^2-4w_+dn}\right)\\
\label{elim4}
\mbox{or }
n_+ & \geq & \frac{1}{2w_+}\left(2\Delta+w_++\sqrt{(2\Delta+w_+)^2-4w_+dn}\right).
\end{eqnarray}
Note that 
$w_+\in [0,1]$ and $dn<\Delta(\Delta+1)$
imply that the roots in (\ref{elim3}) and (\ref{elim4}) are real.
Since $w_{\pm}\geq 0$, it follows using (\ref{elim1}) that 
$$n_+\leq \frac{\Delta+w_+}{w_+}.$$
Since $w_+\in [0,1]$ and $dn<\Delta(\Delta+1)$, it is easy to check that
$$\frac{1}{2w_+}\left(2\Delta+w_++\sqrt{(2\Delta+w_+)^2-4w_+dn}\right)
-\frac{\Delta+w_+}{w_+}>0.$$
This excludes the alternative (\ref{elim4}) and, hence, 
alternative (\ref{elim3}) holds.
Using the conditions $2\Delta\leq dn<\Delta(\Delta+1)$,
it is easy to verify that the lower bound on $n_+$ in (\ref{elim3}) 
is increasing in $w_+\in [0,1]$.
For $w_+=1$, the inequality (\ref{elim3}) yields
$$n_+\leq \Delta+\frac{1}{2}-\frac{1}{2}\sqrt{(2\Delta+1)^2-4dn}$$
and, using (\ref{elim2}), we obtain
\begin{eqnarray*}
\label{elim5}
f\leq \left(2\Delta+1-\sqrt{(2\Delta+1)^2-4dn}\right)(\Delta-d).
\end{eqnarray*}
Next, we suppose that $d_-=\delta=0$.
This implies $w_-=w_{\pm}=0$.
Since ${n_+\choose 2}w_+=m=\frac{dn}{2}$,
we have $w_+=\frac{dn}{n_+(n_+-1)}$.
Since $w_+\leq 1$, 
this implies that $n_+\geq \sqrt{dn+\frac{1}{4}}+\frac{1}{2}$.
Using $n_-=n-n_+$ and $(\ref{elim2b})$, we obtain that 
\begin{eqnarray*}
\label{elim6}
f=2(d-\delta)n_-=2d(n-n_+)\leq d\left(2n-1-\sqrt{4dn+1}\right),
\end{eqnarray*}
which completes the proof of (vi).
\end{proof}
In view of the relation between $s(G)$ and $(P)$,
Lemma \ref{lemma3}(v) implies Theorem \ref{theorem1}
as stated in the introduction,
which is the main result of this section.
It is easy to verify that the bound in Theorem \ref{theorem1} 
outperforms (\ref{e1}) everywhere;
Figure \ref{fig1} illustrates more precisely how the two bounds compare.

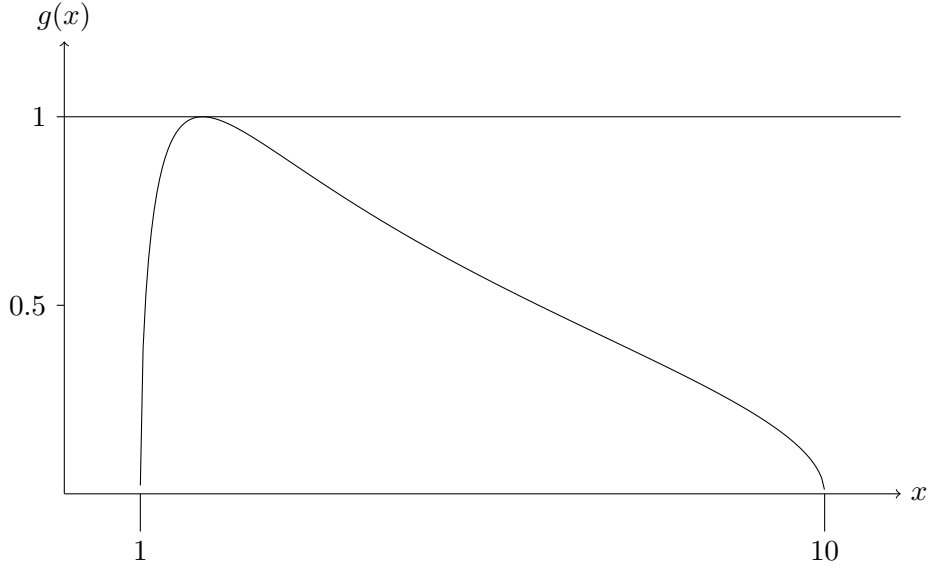
\begin{figure}[H]
\centering
\begin{tikzpicture}[yscale=5,xscale=1]
  \draw[->] (0,0) -- (11, 0) node[right] {$x$};
  \draw[->] (0, 0) -- (0,1.2) node[above] {$g(x)$};
  \draw[-] (0,1) -- (-0.1,1) node[left] {$1$};
  \draw[-] (0,0.5) -- (-0.1,0.5) node[left] {$0.5$};
  \draw[-] (1,0) -- (1,-0.1) node[below] {$1$};
  \draw[-] (10,0) -- (10,-0.1) node[below] {$10$};
  \draw[domain=1.0001:10, samples=250,variable=\x, black] plot ({\x}, {2*(10-\x)*(\x-1)/(9*\x*sqrt((10-\x)*(\x-1)/10))});
  \draw[domain=0:11, smooth, variable=\x, black] plot ({\x},1);
\end{tikzpicture}
\caption{$g(x)=\frac{\frac{2n(\Delta-d)(d-\delta)}{\Delta-\delta}}
{dn\sqrt{\frac{(\Delta-d)(d-\delta)}{\delta\Delta}}}$ as a function of $x=\frac{d}{\delta}\in \left[1,\frac{\Delta}{\delta}\right]$ for $\Delta=10\delta$.}
\label{fig1}
\end{figure}
Choosing $x=(n_+,n_-,w_+,w_-,w_{\pm})$
as 
$$\left(\frac{(d-\delta)n}{\Delta-\delta},
\frac{(\Delta-d)n}{\Delta-\delta},
\frac{\Delta^2}{dn-\Delta},
\frac{\Big(\delta(\Delta-d)n-\Delta(\Delta-\delta)\Big)\delta}{(dn-\Delta)\Big((\Delta-d)n-(\Delta-\delta)\Big)},
\frac{\delta\Delta}{dn-\Delta}\right)$$
or
$$\left(\frac{(d-\delta)n}{\Delta-\delta},
\frac{(\Delta-d)n}{\Delta-\delta},
\frac{\Big(\Delta(d-\delta)n-\delta(\Delta-\delta)\Big)\Delta}{(dn-\delta)\Big((d-\delta)n-(\Delta-\delta)\Big)},
\frac{\delta^2}{dn-\delta},
\frac{\delta\Delta}{dn-\delta}\right)$$
yields a feasible solution for $(P)$
provided that the last three entries $w_+$, $w_-$, and $w_{\pm}$ lie in $[0,1]$.
For fixed $\delta$, $d$, and $\Delta$, 
this is the case provided that $n$ is sufficiently large.
If one or both of these choices is feasible,
then Lemma \ref{lemma3}(v) implies that 
$$s(n,m,\delta,\Delta)=\frac{2n(\Delta-d)(d-\delta)}{\Delta-\delta}$$
and that the feasible solutions are also optimal.
It follows that Theorem \ref{theorem1} is best possible 
for a wide range of values of $(n,m,\delta,\Delta)$
up to terms of smaller order
that are caused by relaxing the integrality conditions.

As illustrated by our next result, which follows from Lemma \ref{lemma3}(vi), 
the optimization problem $(P)$ allows to derive further bounds on the deviation.

\begin{theorem}\label{theorem2}
If $G$ is a graph with 
$n$ vertices, 
average degree $d$, and 
maximum degree at most $\Delta$, where 
$0<d<\Delta<n$, 
$d\leq n-3$, and 
$2\Delta\leq dn<\Delta\left(\Delta+1\right)$, then
$$s(G)\leq 
\max\left\{
\left(2\Delta+1-\sqrt{(2\Delta+1)^2-4dn}\right)(\Delta-d),
d\left(2n-1-\sqrt{4dn+1}\right)
\right\}.$$
\end{theorem}
The maximum in the bound in Theorem \ref{theorem2} reflects
that there are two competing structural options,
which are encoded in the variable values in the proof of Lemma \ref{lemma3} (vi).
Note that the second bound in Theorem \ref{theorem2} 
coincides with Haviland's bound $(\ref{e0})$.

\pagebreak 

\section{Lower bound on $\lambda(G)$ in terms of $n$, $d$, and $s$}\label{sec5}

Throughout this section, 
let $G$ be a graph with 
$n$ vertices, 
average degree $d$, 
degree deviation $s$, and 
spectral radius $\lambda(G)$.
Furthermore, let $s>0$, which implies $d,n>0$.
The results of Section \ref{sec3} imply
$$\lambda(G)\geq \tilde{\lambda}(G)\geq {\rm OPT}(Q_I)$$
for $\tilde{\lambda}(G)$ as in Section \ref{sec3} and 
the following non-linear optimization problem $(Q_I)$ 
depending on the parameters $n$, $d$, and $s$:\\[-5mm]
$$
(Q_I)\hspace{2cm}\begin{array}{lrlll}
\mbox{minimize} &\frac{1}{2}\left(x_++x_-+
\sqrt{\Big(x_+-x_-\Big)^2+4n_+n_-w^2_{\pm}}
\right) & & &\\[4mm]
\mbox{subject to} & n_++n_- & = & n &\\
						& n_+x_++2n_+n_-w_{\pm}+n_-x_- & = & dn & \\
						& n_+(x_+-d)+n_-(d-x_-) & = & s & \\
						& x_++n_-w_{\pm} & \geq & d & \\
						& x_-+n_+w_{\pm} & \leq & d & \\
						& n_+-1-x_+ & \geq & 0 & \\
						& n_--1-x_- & \geq & 0 & \\
						& n_+,n_- & \in & \mathbb{N}_0 &\\
						& x_+,x_- & \geq & 0 &\\
						& w_{\pm} & \in & [0,1], &
\end{array}
$$
where the variables $x_+$ and $x_-$ 
correspond to the quantities $(n_+-1)w_+$ and $(n_--1)w_- $,
respectively.
Since $s>0$, 
every feasible solution for $(Q_I)$ satisfies $n_+,n_-\geq 1$.

Let 
$$f(n_+,n_-,x_+,x_-,w_{\pm})=\frac{1}{2}\left(x_++x_-+
\sqrt{\Big(x_+-x_-\Big)^2+4n_+n_-w^2_{\pm}}\right).$$
Using the three equality constraints from $(Q_I)$
to eliminate $n_-$, $x_-$, and $w_{\pm}$, we obtain
\begin{eqnarray*}
n_- &=& n_-(n_+,x_+)= n-n_+,\\
x_- &=& x_-(n_+,x_+)= \frac{dn+(x_+-2d)n_+-s}{n-n_+},\mbox{ and}\\
w_{\pm}&=& w_{\pm}(n_+,x_+)= \frac{2(d-x_+)n_++s}{2n_+(n-n_+)}.
\end{eqnarray*}
Substituting these expressions, 
the function $f(n_+,n_-,x_+,x_-,w_{\pm})$ becomes
\begin{eqnarray*}
f(n_+,x_+) & = & 
\frac{1}{2(n-n_+)}
\left(
\sqrt{\frac{n}{n_+}\Big((d-x_+)((d-x_+)n+2s)n_++s^2\Big)}
+(d+x_+)n-2dn_+-s
\right)
\end{eqnarray*}
Note that the argument of the root in $f(n_+,x_+)$ arises 
by substituting the expressions for $n_-$, $x_-$, and $w_{\pm}$
into $\Big(x_+-x_-\Big)^2+4n_+n_-w^2_{\pm}$,
which easily implies that it is always positive.
Using the above expressions, 
it is possible to formulate $(Q_I)$ equivalently 
only using the variables $n_+$ and $x_+$.
Unfortunately, the constraints become quite complicated.
Clearly, we have
$0<n_+<n$
and
$0\leq x_+\leq n_+$ for every feasible solution.
Furthermore, the condition $x_-\geq 0$ implies\\[-3mm]
$$x_+\geq L(n_+):=2d-\frac{dn-s}{n_+}.$$
The next lemma
concerns partial derivatives of $f(n_+,x_+)$
at relevant points.

\begin{lemma}\label{lemma4}
Let $n_+$ and $x_+$ be such that $0<n_+<n$ and $0\leq x_+\leq n_+$.
\begin{enumerate}[(i)]
\item $\frac{\partial f(n_+,x_+)}{\partial x_+}\geq 0$.
\item $\frac{\partial f(n_+,0)}{\partial n_+}\leq 0$.
\item If $\frac{\partial f(n_+,L(n_+))}{\partial n_+}=0$, 
then $n_+=\frac{(2dn-3s)n}{2(dn-2s)}$.
\end{enumerate}
\end{lemma}
\begin{proof}
Several times within the proof, we consider the sign of expressions of the following form:
$$\frac{d}{dx}\left(\frac{g_1+\sqrt{g_2}}{g_3}\right)
=\frac{1}{\sqrt{g_2}g_3^2}\left(
\sqrt{g_2}\left(g_1'g_3-g_1g_3'\right)+\frac{1}{2}g_2'g_3-g_2g_3'
\right)$$
for $g_2,g_3>0$.
Note that 
$\frac{d}{dx}\left(\frac{g_1+\sqrt{g_2}}{g_3}\right)\geq 0$
if and only if 
$\sqrt{g_2}\left(g_1'g_3-g_1g_3'\right)+\frac{1}{2}g_2'g_3-g_2g_3'\geq 0$.

\medskip

\noindent (i) A straightforward calculation shows that 
$\frac{\partial f(n_+,x_+)}{\partial x_+}\geq 0$
if and only if
$$\sqrt{\frac{n}{n_+}\Big((d-x_+)((d-x_+)n+2s)n_++s^2\Big)}
-\Big((d-x_+)n+s\Big)\geq 0.$$
Since
\begin{eqnarray*}
\frac{n}{n_+}\Big((d-x_+)((d-x_+)n+2s)n_++s^2\Big)-\Big((d-x_+)n+s\Big)^2=\frac{s^2(n-n_+)}{n_+}\geq 0,
\end{eqnarray*}
this is the case, which completes the proof of (i).

\medskip

\noindent (ii) A straightforward calculation shows that 
$\frac{\partial f(n_+,0)}{\partial n_+}\leq 0$
if and only if
\begin{eqnarray}\label{econd1}
X-Y &\geq & 0\mbox{ for }\\
X &=& (dn+s)n_+^2\sqrt{\frac{n}{n_+}\Big((dn+2s)dn_++s^2\Big)}\mbox{ and }\nonumber\\
Y &=& n\left(d(dn+2s)n_+^2+\frac{3}{2}s^2n_+-\frac{1}{2}ns^2\right).\nonumber
\end{eqnarray}
Note that $X\geq 0$ and that $Y$ is a quadratic function in $n_+$.
If $n_+\leq \frac{-3s^2+s\sqrt{8d^2n^2+16dns+9s^2}}{4(dn+2s)d}$, 
then $Y\leq 0$ and (\ref{econd1}) holds. 
Hence, we may assume that 
\begin{eqnarray}\label{econd2}
n_+\geq \frac{-3s^2+s\sqrt{8d^2n^2+16dns+9s^2}}{4(dn+2s)d}.
\end{eqnarray}
If $X^2-Y^2\geq 0$, then $X=|X|\geq |Y|$ and (\ref{econd1}) holds.
Since
\begin{eqnarray*}
X^2-Y^2 & = & 
s^2n(n-n_+)^2\left(dn_++\frac{s}{2}\right)\left((dn+2s)n_+-\frac{sn}{2}\right),
\end{eqnarray*}
we may therefore assume that
\begin{eqnarray}\label{econd3}
n_+\leq \frac{ns}{2(dn+2s)}.
\end{eqnarray}
Combining (\ref{econd2}) and (\ref{econd3}) yields
$2dn+3s-\sqrt{8d^2n^2+16dns+9s^2}\geq 0$,
which implies the false statement 
$-4d^2n^2-4dns=(2dn+3s)^2-(8d^2n^2+16dns+9s^2)\geq 0$.
This contradiction completes the proof of (ii).
\medskip

\noindent (iii) A straightforward calculation shows that 
$\frac{\partial f(n_+,L(n_+))}{\partial n_+}$ 
equals $\frac{X-Y}{Z}$ for
\begin{eqnarray*}
X&=& (n-n_+)(dn-s)\sqrt{n(n-n_+)\Big((dn-s)^2+d(2s-dn)n_1\Big)},\\
Y&=& dn(n-n_+)^2(dn-2s)+\frac{1}{2}ns^2(2n-3n_+),\mbox{ and}\\
Z &=& 2n_+^2(n-n_+)\sqrt{n(n-n_+)\Big((dn-s)^2+d(2s-dn)n_1\Big)}.
\end{eqnarray*}
Now, let $\frac{\partial f(n_+,L(n_+))}{\partial n_+}=0$.
This implies $X^2-Y^2=0$.
Since
$$X^2-Y^2=s^2nn_+^2
\left(\frac{2dn-s}{2}-dn_+\right)
\left(\frac{(2dn-3s)n}{2}-(dn-2s)n_+\right),$$
we obtain
$n_+\in \left\{ \frac{2dn-s}{2d},\frac{(2dn-3s)n}{2(dn-2s)}\right\}$.
For $n_+=\frac{2dn-s}{2d}$, we obtain
$\frac{\partial f(n_+,L(n_+))}{\partial n_+}
=\frac{4d^2(dn-s)}{(2dn-s)^2}$,
which equals $0$ only if $s=dn$.
Since for $s=dn$, 
we have
$\frac{2dn-s}{2d}=\frac{(2dn-3s)n}{2(dn-2s)}$,
it follows that 
$\frac{\partial f(n_+,L(n_+))}{\partial n_+}=0$
only if $n_+=\frac{(2dn-3s)n}{2(dn-2s)}$,
which completes the proof of (iii).
\end{proof}
A straightforward calculation shows that 
\begin{eqnarray*}
f(n_+,L(n_+)) 
& = & 
\frac{\sqrt{n\Big((dn-s)^2-dn_+(dn-2s)\Big)}}{2n_+\sqrt{n-n_+}}
-\frac{dn-s-2dn_+}{2n_+}\\
& = & 
d+
\frac{s^2}{2(n-n_+)\left(\sqrt{\frac{n\Big((dn-s)^2-dn_+(dn-2s)\Big)}{n-n_+}}+(dn-s)\right)}.
\end{eqnarray*}
Since $\lim\limits_{n_+\to n}\sqrt{n\Big((dn-s)^2-dn_+(dn-2s)\Big)}=s\sqrt{n}>0$,
we obtain 
\begin{eqnarray}\label{3lemma34}
\lim\limits_{n_+\to n}f(n_+,L(n_+))=\infty.
\end{eqnarray}
Furthermore, if $s>dn$, then 
$$\lim\limits_{n_+\to 0}\left(\sqrt{\frac{n\Big((dn-s)^2-dn_+(dn-2s)\Big)}{n-n_+}}+(dn-s)\right)
=\sqrt{(dn-s)^2}+(dn-s)=0,$$
and 
\begin{eqnarray}\label{3lemma34b}
\lim\limits_{n_+\to 0}f(n_+,L(n_+))=\infty
\end{eqnarray}
also in this case.

For Theorem \ref{theorem3},
we consider the following relaxation of $(Q_I)$:
$$
(Q')\hspace{2cm}\begin{array}{lrlll}
\mbox{minimize} & f(n_+,x_+) & & &\\
\mbox{subject to} & x_+ & \geq & L(n_+) & \\
						& x_+ & \leq & n_+ & \\
						& x_+ & \geq & 0 &\\
						& n_+ & < & n & \\
						& n_+ & > & 0. &
\end{array}
$$
Again, it is easy to see that the argument of the root in $f(n_+,x_+)$ 
is positive for every feasible solution of $(Q')$.
We proceed to the proof of the main result of this section.
\begin{proof}[Proof of Theorem \ref{theorem3}]
Let $(n_+,x_+)$ a feasible solution for $(Q')$.
Note that $L(n_+)\geq 0$ if and only if 
$n_+\geq \frac{dn-s}{2d}$.

First, we consider the case that $s<dn$,
which implies, in particular, that $\frac{dn-s}{2d}>0$.
Note that 
\begin{eqnarray}\label{ethm3a}
\frac{d^2n}{\sqrt{d^2n^2-s^2}}\geq \frac{2s}{n},
\end{eqnarray}
because
$\left(\frac{d^2n}{\sqrt{d^2n^2-s^2}}\right)^2-\left(\frac{2s}{n}\right)^2
=\frac{(d^2n^2-2s^2)^2}{(d^2n^2-s^2)n^2}\geq 0.$
By Lemma \ref{lemma4} (i) and (ii),
for $n_+ \leq \frac{dn-s}{2d}$,
we have 
$$f(n_+,x_+)\geq f(n_+,0)\geq f\left(\frac{dn-s}{2d},0\right)=\frac{d^2n}{\sqrt{d^2n^2-s^2}}.$$
Now, let $n_+ \geq \frac{dn-s}{2d}$.
By Lemma \ref{lemma4} (i),
we have 
$f(n_+,x_+)\geq f(n_+,L(n_+))$.

If $s<\frac{dn}{2}$, then $\frac{(2dn-3s)n}{2(dn-2s)}>n$.
If $s=\frac{dn}{2}$, then $\frac{(2dn-3s)n}{2(dn-2s)}$ is undefined.
If $s\in \left(\frac{dn}{2},\frac{dn}{\sqrt{2}}\right]$,
then $\frac{(2dn-3s)n}{2(dn-2s)}\leq \frac{dn-s}{2d}$.
Therefore, for $s\leq \frac{dn}{\sqrt{2}}$,
Lemma \ref{lemma4} (iii) and (\ref{3lemma34}) imply
$$f(n_+,x_+)\geq f(n_+,L(n_+))\geq 
f\left(\frac{dn-s}{2d},L\left(\frac{dn-s}{2d}\right)\right)=\frac{d^2n}{\sqrt{d^2n^2-s^2}}.$$
If $s\in \left(\frac{dn}{\sqrt{2}},dn\right)$,
then Lemma \ref{lemma4} (iii) and (\ref{3lemma34}) imply
\begin{eqnarray*}
f(n_+,x_+) & \geq & f(n_+,L(n_+))\\
&\geq &
\min\left\{
f\left(\frac{(2dn-3s)n}{2(dn-2s)},L\left(\frac{(2dn-3s)n}{2(dn-2s)}\right)\right),
f\left(\frac{dn-s}{2d},L\left(\frac{dn-s}{2d}\right)\right)
\right\}\\
&\geq &
\min\left\{
\frac{2s}{n},
\frac{d^2n}{\sqrt{d^2n^2-s^2}}
\right\}\\
&\stackrel{(\ref{ethm3a})}{\geq} &
\frac{2s}{n}.
\end{eqnarray*}
Altogether, for $s<dn$, the desired statement follows.

Next, let $s=dn$.
In this case, we have $n\geq n_+\geq x_+\geq L(n_+)=2d$, and 
Lemma \ref{lemma4} (i) implies
$$f(n_+,x_+)\geq f(n_+,2d)
=d+\frac{dn}{2\sqrt{n_+(n-n_+)}}
\geq d+\frac{dn}{2\sqrt{\frac{n}{2}\left(n-\frac{n}{2}\right)}}
=2d=\frac{2s}{n}.$$
%\begin{cases}
%2d=\frac{2s}{n}, & \mbox{ if } d\leq \frac{n}{4}\mbox{ and}\\[3mm]
%d+\frac{dn}{2\sqrt{2d(n-2d)}}, & 
%\mbox{ if } d\in \left(\frac{n}{4},\frac{n}{2}\right).
%\end{cases}
%$$
%Note that $2d\leq n$ in this case.

Finally, let $s>dn$.
By Lemma \ref{lemma4} (iii), (\ref{3lemma34}), and (\ref{3lemma34b})
we obtain
\begin{eqnarray*}
f(n_+,x_+) & \geq & f(n_+,L(n_+))
\geq f\left(\frac{(2dn-3s)n}{2(dn-2s)},L\left(\frac{(2dn-3s)n}{2(dn-2s)}\right)\right)
=\frac{2s}{n},
\end{eqnarray*}
which completes the proof.
\end{proof}

\begin{corollary}\label{corollary1}
Let $G$ be a graph with 
$n$ vertices, 
average degree $d$, and 
degree deviation $s$ with $s>0$.

If 
either $s>\frac{dn}{\sqrt{2}}$
or $s\leq \frac{dn}{\sqrt{2}}$ and $d\leq \frac{n}{2}$, then
$$\lambda(G)-d\geq\tilde{\lambda}(G)-d\geq \frac{s^2}{n^2\sqrt{2dn}}.$$
\end{corollary}
\begin{proof}
First, we suppose that $d\leq \frac{n}{2}$ and $s\leq \frac{dn}{\sqrt{2}}$.
Since
\begin{eqnarray*}
\frac{\partial}{\partial s}
\left(\frac{d^2n}{\sqrt{d^2n^2-s^2}}-d\right)
&=&\frac{d^2ns}{(d^2n^2-s^2)^{\frac{3}{2}}}>0,\\
\frac{\partial}{\partial s}
\left(\frac{s^2}{n^2\sqrt{2dn}}\right)
&=& \frac{2s}{n^2\sqrt{2dn}}>0,\mbox{ and}\\
\left(\frac{d^2ns}{(d^2n^2-s^2)^{\frac{3}{2}}}\right)^2
-\left(\frac{2s}{n^2\sqrt{2dn}}\right)^2
&=&\frac{2s^2\left(d^5n^6\left(\frac{n}{2}-d\right)+s^2\left(\left(\frac{3d^2n^2}{2}-s^2\right)^2+\frac{3}{4}d^4n^4\right)\right)}{(d^2n^2-s^2)^3n^5d}>0,
\end{eqnarray*}
we obtain
\begin{eqnarray*}
\left(\frac{d^2n}{\sqrt{d^2n^2-s^2}}-d\right)
-\left(\frac{s^2}{n^2\sqrt{2dn}}\right)
&\geq & 
\left(\frac{d^2n}{\sqrt{d^2n^2-0^2}}-d\right)
-\left(\frac{0^2}{n^2\sqrt{2dn}}\right)=0.
\end{eqnarray*}
In view of Theorem \ref{theorem3}, 
this implies the desired statement for 
$d\leq \frac{n}{2}$ and $s\leq \frac{dn}{\sqrt{2}}$.

Next, we suppose that $s>\frac{dn}{\sqrt{2}}$.
If $d\leq \frac{n-1}{2}$, then 
$$s
\stackrel{(\ref{e0})}{\leq}
d\left(2n-1-\sqrt{4nd+1}\right)
\leq 2dn
\leq \sqrt{\frac{2n}{d}}dn
=\sqrt{2dn}n$$
and if $d\geq \frac{n-1}{2}$, then $\psi=n-1-d\leq \frac{n-1}{2}\leq d$ and
$$s
\stackrel{(\ref{e0})}{\leq}
\psi\left(2n-1-\sqrt{4n\psi+1}\right)
\leq \sqrt{2\psi n}n
\leq \sqrt{2dn}n.$$
Altogether, we obtain $s\leq \sqrt{2dn}n$, which implies
\begin{eqnarray*}
\frac{\partial}{\partial s}\left(\left(\frac{2s}{n}-d\right)-\frac{s^2}{n^2\sqrt{2dn}}\right)
&=& \frac{2}{n}-\frac{\sqrt{2}s}{n^2\sqrt{dn}}\geq 0.
\end{eqnarray*}
Since $s\geq \frac{dn}{\sqrt{2}}$, we obtain
\begin{eqnarray*}
\left(\frac{2s}{n}-d\right)-\frac{s^2}{n^2\sqrt{2dn}}
&\geq & 
\left(\frac{2\left(\frac{dn}{\sqrt{2}}\right)}{n}-d\right)-\frac{\left(\frac{dn}{\sqrt{2}}\right)^2}{n^2\sqrt{2dn}}
=
\left(\sqrt{2}-1\right)d-\frac{\sqrt{2}d^2}{4\sqrt{dn}}
\geq
\left(\sqrt{2}-1-\frac{\sqrt{2}}{4}\right)d>0.
\end{eqnarray*}
In view of Theorem \ref{theorem3}, 
this implies the desired statement for $s\geq \frac{dn}{\sqrt{2}}$.
\end{proof}

\section{Proof and extremal graphs for (\ref{e1})}\label{sec2}

Throughout this section, let $G$ be a graph with 
$n$ vertices,
$m$ edges,
minimum degree $\delta>0$, and 
maximum degree $\Delta$.
Let $V$ denote the vertex set of $G$ and 
let $d$ denote the average degree $\frac{2m}{n}$ of $G$.

The two vectors 
\begin{eqnarray*}
x&=&(x_u)_{u\in V}=\left(\frac{|d_G(u)-d|}{\sqrt{d_G(u)}}\right)_{u\in V}
\,\,\,\,\,\,\,\,\mbox{ and}\,\,\,\,\,\,\,\,
y=(y_u)_{u\in V}=\left(\sqrt{d_G(u)}\right)_{u\in V}
\end{eqnarray*}
satisfy
\begin{eqnarray*}
||x||_2^2 
& =& \sum\limits_{u\in V}\frac{(d_G(u)-d)^2}{d_G(u)}
=\underbrace{\sum\limits_{u\in V}d_G(u)}_{2m}
-\underbrace{\sum\limits_{u\in V}2d}_{4m}+\sum\limits_{u\in V}\frac{d^2}{d_G(u)}
= d^2\left(\sum\limits_{u\in V}\frac{1}{d_G(u)}\right)-2m\mbox{ and}\\
||y||_2^2 &=& 2m.
\end{eqnarray*}
If $\alpha$ denotes the angle between $x$ and $y$ in $\mathbb{R}^V$, 
then 
\begin{eqnarray}
s(G) &=& \sum\limits_{u\in V}|d_G(u)-d|\nonumber\\
&=&x^Ty\nonumber\\
&=&||x||_2\cdot ||y||_2\cdot \cos(\alpha)\nonumber\\
&\leq &||x||_2\cdot ||y||_2\nonumber\\
&=& \sqrt{\left(d^2\left(\sum\limits_{u\in V}\frac{1}{d_G(u)}\right)-2m\right)2m}\nonumber\\
&=& d\sqrt{2m\left(\sum\limits_{u\in V}\frac{1}{d_G(u)}\right)-n^2}\label{ea1}
\end{eqnarray}
with equality if and only if $\alpha=0$,
that is, the two vectors $x$ and $y$ are parallel.
By the definition of $x$ and $y$, we have that $\alpha=0$  
if and only if the expression
$\frac{|d_G(u)-d|}{d_G(u)}$
is constant on $V$.

A simple argument using the convexity of the function $z\mapsto \frac{1}{z}$ implies that 
\begin{eqnarray}
\sum\limits_{u\in V}\frac{1}{d_G(u)}
\leq \frac{n(\Delta+\delta)-2m}{\delta\Delta}\label{ea2}
\end{eqnarray}
with equality if and only if all vertex degrees are in $\{ \delta,\Delta\}$.

Combining (\ref{ea1}) and (\ref{ea2}) yields (\ref{e1})
with equality if and only if 
\begin{itemize}
\item $\frac{|d_G(u)-d|}{d_G(u)}$
is constant on $V$
and 
\item all vertex degrees are in $\{ \delta,\Delta\}$.
\end{itemize}
This implies that $\frac{\Delta-d}{\Delta}=\frac{d-\delta}{\delta}$,
which implies $d=\frac{2\delta\Delta}{\Delta+\delta}$
and, thus,
$m=\frac{dn}{2}=\frac{\delta\Delta n}{\Delta+\delta}$.
In particular, for $\delta<\Delta$, there are non-bipartite extremal graphs.

\newpage

\appendix

\section{Appendix: Nikiforov's lower bound for $d>n/2$}

As in the previous section, let $G$ be a graph with 
$n$ vertices, 
average degree $d$, 
degree deviation $s$, and 
spectral radius $\lambda=\lambda(G)$.
Again, let $s>0$, which implies $d,n>0$. In this section, we show Nikiforov's conjectured lower bound.

\begin{theorem}\label{niki}
$\lambda \ge d + \frac{\sqrt{2} s^{2}}{2 n^{2} \sqrt{d n}}$
\end{theorem}

We already showed Theorem \ref{niki} for $d\le n/2$ (see Corollary \ref{corollary1}).
Therefore, for the rest of this section, let $d>\frac{n}{2}$. Using (\ref{e0}) or (\ref{havi2}) below, it is a mathematical standard task to show $s<0.35n^2$.
We show the following two theorems.

\begin{theorem}\label{theoremF1}
$\lambda \ge d + \frac{s^{2}}{n^{3}} -\frac{ 2 (d + n) s^4}{n^{8}}$ for $d>\frac{n}{2}$
\end{theorem}

\begin{theorem}\label{theoremF2}
    $\lambda \ge d + \frac{4 s^{2}}{n \left(3n-2d\right)\left(2d+n\right)} - \frac{24 \left(2 d - n\right)s^3}{n^{2} \left(3 n-2d\right)^{2} \left(2 d + n\right)^{2}}$ for $\frac{n}{2}<d\le 0.8n$ and $s\ge \frac{7n\left(d-\frac{n}{2}\right)}{10}$  
\end{theorem}

Note that Theorem \ref{theoremF2} was constructed to show Nikiforov's Conjecture where Theorem \ref{theoremF1} is not strong enough and leaves room for improvement.\\

Again, we consider $(Q_I)$.
The condition $w_3\le 1$ yields 
$$x_1\ge L_2(n_+):=d - n + n_+ + \frac{s}{2n_+}. $$
Remember $x_-\geq 0$ implies
$x_+\geq L(n_+)=2d-\frac{dn-s}{n_+}$.
Comparing $L$ and $L_2$ shows $L_2\ge L$ if and only if $$n_+\le N:= \frac{d}{2} + \frac{n}{2} - \frac{\sqrt{d^2 - 2dn + n^2 + 2s}}{2}.$$
Since $x_+ \ge \max\left\{ L(n_+),L_2(n_+)\right\}$, we consider the
following relaxations of $(Q_I)$. Note that $\lambda \ge OPT(Q') \ge  \min\{ OPT(Q''),OPT(Q''') \} $.

$$
(Q'')\hspace{0cm}\begin{array}{lrlll}
\mbox{minimize} & f(n_+,x_+) & & &\\
\mbox{subject to} & x_+ & \geq & L_2(n_+) & \\
						& x_+ & \leq & n_+ & \\
						& x_+ & \geq & 0 &\\
						& n_+ & \le & N & \\
%						& n_+ & \le & \frac{d}{2} + \frac{n}{2} - \frac{\sqrt{d^2 - 2dn + n^2 + 2s}}{2} & \\
						& n_+ & > & 0 &
\end{array}
%$$
%
%and
\hspace{0.5cm}\text{and}\hspace{1cm}
%
%$$
(Q''')\hspace{0cm}\begin{array}{lrlll}
\mbox{minimize} & f(n_+,x_+) & & &\\
\mbox{subject to} & x_+ & \geq & L(n_+) & \\
						& x_+ & \leq & n_+ & \\
						& x_+ & \geq & 0 &\\
						& n_+ & \ge & N & \\
%						& n_+ & \ge & \frac{d}{2} + \frac{n}{2} - \frac{\sqrt{d^2 - 2dn + n^2 + 2s}}{2} & \\
						& n_+ & > & 0 &
\end{array}
$$
\medskip

\begin{lemma}\label{lemmaF5}
$L(N)\le N-1$ if and only if $s\le (n-d-1)\left(2n-1-\sqrt{4n(n-d-1)+1}\right)$.
\end{lemma}

\begin{proof}
$$
L(N)-(N-1)
=\frac{ -(n+d)(n-d-1) + s+ (n-d-1) \sqrt{n^{2} - 2 d n + d^{2} + 2 s}}{d + n - \sqrt{n^{2} - 2 d n + d^{2} + 2 s}}
$$
Note that %$s\le (n+d)(n-d-1)$ since 
(\ref{e0}) implies $s\le (n-d-1)\left(2n-1-\sqrt{4n(n-d-1)+1}\right) \le 
(n-d-1)\left(2n-1-(n-d-1)\right)=(n-d-1)(n+d)$. Hence $L(N)-(N-1)=0$ if and only if 
$$ -(-(n+d)(n-d-1) + s)^2+ (n-d-1)^2 \left(d^{2} - 2 d n + n^{2} + 2 s \right)=0, $$
which, as a function in $s$, has the solutions $ \left(n-d - 1\right) \left(2 n -1 \pm \sqrt{4n(n-d-1)+1}\right)  $, where only the smaller solution is valid according to (\ref{e0}). This completes the proof.
\end{proof}

\begin{lemma}\label{lemmaF4}
$(n_+,x_+)=(N,L(N))$ is a feasible solution for $(Q'')$ and an optimal solution for $(Q''')$. 
\end{lemma}

\begin{proof}
Lemma \ref{lemmaF5} 
and
($\ref{e0}$) imply $L(N)=L_2(N)\le N-1$.
Since $0< N <n$, 
$(n_+,x_+)=(N,L(N))$ is a feasible solution for $(Q'')$ and $(Q''')$.
Since $\frac{dn-s}{2d}\le N$, it follows analogously to the proof of Theorem \ref{theorem3} that $f(n_+,x_+)\ge f(N,L(N))$, which completes the proof.
\end{proof}

Since $L(N)-(N-1)> L(N)-N$, we receive together with (\ref{e0}) for the corresponding zeros
\begin{align}
s\le (n-d-1)\left(2n-\sqrt{4n(n-d-1)+1} -1 \right) \le (n-d)\left(2n-\sqrt{4n(n-d)}\right).\label{havi2}
\end{align}

Since $L(N)=L_2(N)$, Lemma \ref{lemmaF4} implies $OPT(Q''')\ge OPT(Q'')$ and we receive $\lambda \ge OPT(Q'')$. Therefore, in order to find new bounds on $\lambda$ depending on $n,m,d>\frac{n}{2}$ and $s$, one simply has to bound $OPT(Q'')$
as tight as possible.\\

Lemma \ref{lemma4} implies $\frac{\partial f(n_+,x_1)}{\partial x_1}\ge 0$ 
%(Lemma \ref{lemma4}) 
even if $x_+>n_+$ and hence $\lambda \ge \min\limits_{0< n_+< n} f(n_+,L_2(n_+))$.\\
Note that $(n^*,L_2(n^*))$, where $n^*=\arg\min_{0< n_+< n} f(n_+,L_2(n_+))$, is not necessarily a feasible solution for $(Q'')$.
Simplifying $f(n_+,L_2(n_+))$ yields
{\small
\begin{align}
&f(n_+,L_2(n_+))\\
&=\frac{n \left(4 d n_+ - 2 n n_+ + 2 n_+^{2} + s\right) - n_+ \left(4 d n_+ + 2 s\right) + \sqrt{n \left(4 n_+ s^{2} + \left(n \left(2 n n_+ - 2 n_+^{2} - s\right) + 4 n_+ s\right) \left(2 n n_+ - 2 n_+^{2} - s\right)\right)}}{4 n_+ \left(n - n_+\right)}.\label{eqF1}
\end{align}
}%

Again, it is easy to see that the argument of the root in $(\ref{eqF1})$ is strictly bigger 0 if $0<n<n_+$, since it is received through substitution -- and pulling strictly positive terms out of the root -- from $\big(x_+-x_-\big)^2+4n_+n_-w^2_{\pm}$.
The same argumentation holds throughout the paper.

\begin{lemma}\label{lemmaF1}
Let $p$ be the unique zero of the polynomial $P(n_+):=-2n^3 + 8n^2n_+ - 12nn_+^2 - ns + 8n_+^3$.
$$f(n_+,L_2(n_+))\ge f(p,L_2(p)) \text{ for }0<n_+<n$$
\end{lemma}

Considering the derivative $\frac{\partial P}{\partial n_+}=8 \left(n^{2} - 3 n n_+ + 3 n_+^{2}\right)>0
$ shows that $P(n_+)$ is strictly monotone increasing in $n_+$ everywhere, implying $P$ has at most one zero. 
We consider
\begin{align*}
&p_1:=\frac{n}{2}+\frac{s}{2n}-\frac{s^3}{2n^5} \text{ and }\\
&p_2:=\frac{n}{2}+\frac{s}{2n}, 
\end{align*}
where 
$P(p_1)=\frac{s^5(-3n^8 + 3n^4s^2 - s^4)}{n^{15}}< 0$ and $P(p_2)=\frac{s^3}{n^3}> 0$.
Hence $P$ has a unique zero, which we call $p$ and $p_1< p < p_2$. It is easy to see that $0<p_1<p_2<n$. Now we show Lemma \ref{lemmaF1}.

\begin{proof}
Again, using 
$$\frac{d}{dx}\left(\frac{g_1+\sqrt{g_2}}{g_3}\right)
=\frac{1}{\sqrt{g_2}g_3^2}\left(
\sqrt{g_2}\left(g_1'g_3-g_1g_3'\right)+\frac{1}{2}g_2'g_3-g_2g_3'
\right),$$
we receive $\frac{\partial f(n_+,L_2(n_+))}{\partial n_+}=0$ if and only if $X+Y=0$ for
$$X=4 s \left(- n^{2} + 2 n n_+ - 2 n_+^{2}\right) \sqrt{n \left(4 n_+ s^{2} + \left(n \left(2 n n_+ - 2 n_+^{2} - s\right) + 4 n_+ s\right) \left(2 n n_+ - 2 n_+^{2} - s\right)\right)} $$
and
$$
Y=4ns(2n_+(n-n_+)(n^2-2nn_++2n_+^2)+s(2nn_+-n^2)).
$$

Note that $X<0$ and hence $X^2+Y^2>0$. 
$$-X^2+Y^2=64 n n_+^{2} s^{3} \left(n - n_+\right)^{2} \left(-2 n^{3} + 8 n^{2} n_+ - 12 n n_+^{2} - n s + 8 n_+^{3}\right)$$

For $n_+\ge \frac{n}{2}$, it follows 
$Y\ge 0$.
%=4ns(2n_+(n-n_+)(n^2-2nn_++2n_+^2)+s(2nn_+-n^2))\ge 0$.
%$Y=-4ns(2n_+(n-n_+)(-n^2+2nn_+-2n_+^2)+s(n^2-2nn_+))\ge 0$.\\ 
Since $\frac{n}{2}\le p_1\le p \le p_2$, $p$ is indeed a zero of $\frac{\partial f(n_+,L_2(n_+))}{\partial n_+}$. Further more, %}
%$\left(-X^2+Y^2\right)|_{n_+=p_1}<0$ and $\left(-X^2+Y^2\right)|_{n_+=p_2}>0$ implies $p$ is a minimum.}
%
\begin{align*}
\left(-X^2+Y^2\right)\big|_{n_+=p_1}&=
- \frac{4 s^{8} \left(n^{6} +n^{4} s - s^{3}\right)^{2} \left(n^{6} - n^{4} s + s^{3}\right)^{2} \left(3 n^{8} - 3 n^{4} s^{2} + s^{4}\right)}{n^{34}}<0 \\
\left(-X^2+Y^2\right)\big|_{n_+=p_2}&=
\frac{4 s^{6} \left( n^{2}- s\right)^{2} \left(n^{2} + s\right)^{2}}{n^{6}}>0,
\end{align*}
implying $p$ is a minimum.
\end{proof}

%?? Comment: 
Note that $(n_+,x_+)=(p,L_2(p))$ is not a feasible solution of $(Q'')$ for all $d$ and $s$. Two errors can occur: First, $L_2(p)> p$ and second $p> N$.

\begin{lemma}
$$f(p,L_2(p))\ge f_1:=
\frac{n^{8} \left(d n^{4} - d s^{2} - \frac{n^{5}}{2} - \frac{n s^{2}}{2} + \frac{n}{2} \sqrt{n^{8} + 6 n^{4} s^{2} - 3 s^{4}}\right)}{n^{12} - n^{8} s^{2} + 2 n^{4} s^{4} - s^{6}}
$$
\end{lemma}
\begin{proof} 
%
%Assume $d\ge n/2$.\\
We consider the counter and denominator of $f(n_+,L_2(n_+))$, more precisely $(\ref{eqF1})$, separately.\\
Let $f_u(n_+):=f(n_+,L_2(n_+))(4n_+(n-n_+))$ and $$w_u:=\sqrt{n \left(4 n^{3} n_+^{2} - 8 n^{2} n_+^{3} - 4 n^{2} n_+ s + 4 n n_+^{4} + 12 n n_+^{2} s + n s^{2} - 8 n_+^{3} s\right)}.$$
%\\
$$\frac{\partial f_u}{\partial n_+}=\frac{1}{w_u}\left(4 n^{4} n_+ - 12 n^{3} n_+^{2} - 2 n^{3} s + 8 n^{2} n_+^{3} + 12 n^{2} n_+ s - 12 n n_+^{2} s + w_u \left(2 \left(2 d - n\right) \left(n - 2 n_+\right) -2s\right)\right)$$

Claim: $\frac{\partial f_u}{\partial n_+}<0$ for $n_+\in [p, p_2]$\\
Since $2 \left(2 d - n\right) \left(n - 2 n_+\right) -2s< 0$ for $n_+\ge \frac{n}{2}$, it follows $\frac{\partial f_u}{\partial n_+}<0$ for $n_+\in [p, p_2]$ if 
$$R:= 4 n^{4} n_+ - 12 n^{3} n_+^{2} - 2 n^{3} s + 8 n^{2} n_+^{3} + 12 n^{2} n_+ s - 12 n n_+^{2} s \le 0.$$
Note that $P(n_+)\ge 0$ for $n_+\ge p$. Hence for $n_+\ge p$, 
$$R\le R+ n^2P=- n \left(2 n_+-n\right)^{2} \left(2 n^{2} - 4 n n_+ + 3 s\right),$$ which is smaller or equal zero if $n_+\le \frac{n}{2}+\frac{3s}{4n}$. This completes the proof of the claim.\\

Now $\frac{\partial f_u}{\partial n_+}< 0$ for $n_+\in [p,p_2]$ implies $f_u(p)\ge f_u(p_2) $. For the denominator of $f(n_+,L_2(n_+)) $ follows since $\frac{n}{2}\le p_1 \le p<n$ that $4 p \left(n - p\right)\le 4p_1 (n-p_1)$. Together, we receive
%
%$$
$$f(p,L_2(p))=\frac{f_u(p)}{4p(n-p)} \ge\frac{f_u(p_2)}{4p_1(n-p_1)} =\frac{n^{8} \left(d n^{4} - d s^{2} - \frac{n^{5}}{2} - \frac{n s^{2}}{2} + \frac{n}{2} \sqrt{n^{8} + 6 n^{4} s^{2} - 3 s^{4}}\right)}{n^{12} - n^{8} s^{2} + 2 n^{4} s^{4} - s^{6}}.
$$
\end{proof}

Note that
$
n^{12} - n^{8} s^{2} + 2 n^{4} s^{4} - s^{6}=\left(n^{6} +n^{4} s - s^{3}\right) \left(n^{6} - n^{4} s + s^{3}\right)
>0$.

\begin{lemma}
$w:= \sqrt{n^{8} + 6 n^{4} s^{2} - 3 s^{4}}\ge w_{1}:= n^{4} + 3 s^{2} - \frac{6 s^{4}}{n^{4}} + \frac{18 s^{6}}{n^{8}} - \frac{72 s^{8}}{n^{12}}$
\end{lemma}

\begin{proof}$$w^2-w_1^2=\frac{108 s^{10} \left(6 n^{12} - 11 n^{8} s^{2} + 24 n^{4} s^{4} - 48 s^{6}\right)}{n^{24}}>0,$$
since
$
2 n^{12} - 11 n^{8} s^{2} + 24 n^{4} s^{4} - 48 s^{6}=
\left(n^{2} - 2 s\right) \left(n^{2} + 2 s\right) \left(2 n^{8} - 3 n^{4} s^{2} + 12 s^{4}\right)>0$
\end{proof}

In $f_1$, we replace $w$ by $w_1$.
$$f_1\ge f_2:=\frac{n^{8} \left(d n^{4} - d s^{2} - \frac{n^{5}}{2} - \frac{n s^{2}}{2} + \frac{n}{2} \left(n^{4} + 3 s^{2} - \frac{6 s^{4}}{n^{4}} + \frac{18 s^{6}}{n^{8}} - \frac{72 s^{8}}{n^{12}}\right)\right)}{n^{12} - n^{8} s^{2} + 2 n^{4} s^{4} - s^{6}}
$$
$$=\frac{ d n^{15} - d n^{11} s^{2} + n^{12} s^{2} - 3 n^{8} s^{4} + 9 n^{4} s^{6} - 36 s^{8}}{n^{3} \left(n^{12} - n^{8} s^{2} + 2 n^{4} s^{4} - s^{6}\right)}
$$\\

\begin{lemma}\label{lemmaf2}
$f_2 \ge f_3:=d + \frac{s^{2}}{n^{3}} -\frac{ \left(2 (d + n)\right) s^4}{n^{8}}$
\end{lemma}

\begin{proof}
    $$f_{2}-f_3
     =\frac{s^{6} \left( 5 n^{9} - 31 n^{5} s^{2} - 2 n s^{4}
-d( n^{8} -4 n^{4} s^{2} + 2 s^{4}) \right)}{n^{8} \left(n^{12} - n^{8} s^{2} + 2 n^{4} s^{4} - s^{6}\right)}$$
Note
$n^{8} -4 n^{4} s^{2} + 2 s^{4}>0$ implies
%= 
\begin{align*}
&5 n^{9} - 31 n^{5} s^{2} - 2 n s^{4}
-d( n^{8} -4 n^{4} s^{2} + 2 s^{4})
\ge \left( 5 n^{9} - 31 n^{5} s^{2} - 2 n s^{4}
-d( n^{8} -4 n^{4} s^{2} + 2 s^{4}) \right)\big|_{d=n}\\
&=4n^9-27n^5s^2-4ns^4 \ge \left( 4n^9-27n^5s^2-4ns^4 \right)\big|_{s=0.35n^2}>0.6n^9.
\end{align*}
\end{proof}

This completes the proof of Theorem \ref{theoremF1}. \\

As mentioned before, Theorem \ref{theoremF1} implies parts of Theorem \ref{niki}.

\begin{lemma}\label{lemmaN}
$\lambda \ge d + \frac{\sqrt{2} s^{2}}{2 n^{2} \sqrt{d n}}$ for $d\ge 0.8n$ or $s\le s_0$ (and $d\ge \frac{n}{2})$, where $$s_0:= \frac{1}{2}\sqrt{\frac{n^{5} \left(2 \sqrt{d n}-\sqrt{2} n\right)}{\sqrt{d n} \left(d + n\right)}}. $$
\end{lemma}

\begin{proof}
Note $s_0$ is real since $ 2 \sqrt{d n}> \sqrt{2}n$ for $d>\frac{n}{2}$.
Lemma \ref{lemmaf2} implies
$$\lambda - d-  \frac{\sqrt{2} s^{2}}{2 n^{2} \sqrt{d n}} \ge d + \frac{s^{2}}{n^{3}} -\frac{ \left(2 (d + n)\right) s^4}{n^{8}}- \left( d+ \frac{\sqrt{2} s^{2}}{2 n^{2} \sqrt{d n}}\right)=\frac{s^{2} \left(-4 (d+n) \sqrt{d n}s^2 - \sqrt{2} n^{6} + 2 n^{5} \sqrt{d n}\right)}{2 n^{8} \sqrt{d n}},$$
which, as a function in $s$, has the zeros $\pm s_0$
and therefore implies the result for $s\le s_0$. \\

It remains to show that $s\le s_0$ for $d\ge 0.8n$.\\
Let $d\ge 0.8n$. 
(\ref{havi2}) implies
$s\le h:=\left( n-d\right) \left(2 n - 2 \sqrt{n \left( n-d\right)}\right)$.

$$\frac{\partial h}{\partial d}=-2 n + 3 \sqrt{n \left(n-d\right)}\le \left( -2n+3\sqrt{n(n-d)}\right)|_{d=0.8n}<0$$
$$\frac{\partial {s_0}^2}{\partial d}=\frac{n^{3} \left(- 4 d^{3} n^{2} + \sqrt{2} n^{2} \left(d n\right)^{\frac{3}{2}} + 3 \sqrt{2} \left(d n\right)^{\frac{5}{2}}\right)}{8 d^{3} \left(d^{2} + 2 d n + n^{2}\right)}>0 \text{ since }3\sqrt{2}>4$$
Now, a simple calculation shows $h|_{d=0.8n}<s_0|_{d=0.8n}$, which completes the proof.
\end{proof}
\medskip
\medskip

Next we proof Theorem \ref{theoremF2}.\\
Since $p$ is the unique minimum of $f(n_+,L_2(n_+))$, if $N \le p$, then $\lambda \ge OPT(Q'')\ge f(N,L_2(N))$. 
Further more, for every $N'$ with $N\le N'\le p$ follows $\lambda \ge OPT(Q'')\ge f(N',L_2(N'))$. $N$ is quite complicated as it contains a root. Therefore, to bound $\lambda$, we are looking for an approximation of $N$ between $N$ and $p$. The following lemma realizes this idea. 

\newcommand{\sa}{ t }

\begin{lemma}\label{lemmaF3}
Let $\sa:=\frac{7n}{10} \left(d - \frac{n}{2}\right)$ and
$N_{1}:=\frac{d}{2} + \frac{n}{4} - \frac{s}{4 n}$.
\begin{enumerate}[(i)]
 %   \item $\sa\le s_0$
    \item $N\le N_{1}\le p_1 \le p$ for $s\ge t$ and $\frac{n}{2}\le d\le 0.8n$
    \item $\lambda \ge OPT(Q'')\ge f(N_{1},L_2(N_{1}))$ for $s\ge t$ and $\frac{n}{2}< d \le 0.8n$
\end{enumerate}
\end{lemma}
\begin{proof}
(i) $$(p_1-N_{1})|_{s=\sa}=\frac{\left(2 d - n\right) \left(-1372 d^{2} + 1372 d n - 143 n^{2}\right)}{16000 n^{2}}>0,$$
since, as a function in $d$, $-1372 d^{2} + 1372 d n - 143 n^{2}$ has the zeros 
$\frac{n \left(49 \pm 10 \sqrt{14}\right)}{98}$ and hence 
has no zero for $\frac{n}{2}< d \le 0.8n$.
$$(N-N_{1})|_{s=t}=
\frac{7 d}{40} + \frac{13 n}{80} - \frac{\sqrt{100 d^{2} - 60 d n + 30 n^{2}}}{20}<0,
$$ \\
since $\left(\frac{7 d}{40} + \frac{13 n}{80}\right)^2- \left( \frac{\sqrt{100 d^{2} - 60 d n + 30 n^{2}}}{20} \right)^2=
-\frac{\left(2 d - n\right) \left(702 d - 311 n\right)}{6400}<0
$.

Now we consider the derivative of $N,N_{1}$ and $p$ after $s$.
Since $s\le 0.5 n^2$ and $d\ge 0.5n$,
    $$ \frac{\partial N}{\partial s}= - \frac{1}{2 \sqrt{d^{2} - 2 d n + n^{2} + 2 s}} \le - \frac{\sqrt{5}}{5n} < -\frac{1}{4n} = \frac{\partial N_{1}}{\partial s}.$$

    Also, note that $\frac{\partial p_1}{\partial s}=\frac{1}{2 n} - \frac{3 s^{2}}{2 n^{5}}\ge 0$ if $s\le \frac{\sqrt{3}}{3}n^2$, which is true.
    Together, we receive $ \frac{\partial N}{\partial s}\le \frac{\partial N_{1}}{\partial s} \le \frac{\partial p_1}{\partial s}$, which completes the proof.\\

    (ii) follows from (i) as explained before the lemma.
\end{proof}

Now we bound $f(N_{1},L_2(N_{1}))$ step by step using Taylor expansions.

{\footnotesize
\setlength{\thinmuskip}{1.5mu} % Default is 3mu
\setlength{\medmuskip}{2mu} % Default is 4mu
\setlength{\thickmuskip}{2.5mu} % Default is 5mu
\begin{align*}
&f(N_1,L_2(N_1))\\
&=\frac{n \left(4 d n_+ - 2 n n_+ + 2 n_+^{2} + s\right) - n_+ \left(4 d n_+ + 2 s\right) + \sqrt{n \left(4 n_+ s^{2} + \left(n \left(2 n n_+ - 2 n_+^{2} - s\right) + 4 n_+ s\right) \left(2 n n_+ - 2 n_+^{2} - s\right)\right)}}{4 n_+ \left(n - n_+\right)}\Bigg|_{n_+=\frac{d}{2} + \frac{n}{4} - \frac{s}{4 n}}  
\end{align*}
}%
We first consider
$$z:=\sqrt{n \left(4 n_+ s^{2} + \left(n \left(2 n n_+ - 2 n_+^{2} - s\right) + 4 n_+ s\right) \left(2 n n_+ - 2 n_+^{2} - s\right)\right)}\Big|_{n_+=\frac{d}{2} + \frac{n}{4} - \frac{s}{4 n}}.$$

\begin{lemma}
$$z \ge z_1:= \frac{n \left( 3 n-2d\right) \left(2 d + n\right)}{8} + \left(\frac{3 d}{2} - \frac{3 n}{4}\right)s + \frac{3}{8 n}s^2 - \frac{4 \left(2 d - n\right)}{n^{2} \left(3 n-2d\right) \left(2 d + n\right)}s^3$$

\end{lemma}

\begin{proof}
$$ z^2-z_{1}^2=   \frac{s^{4} \left(2 d - n\right) \left(-48 d^{3} n^{2} + 72 d^{2} n^{3} + 12 d n^{4} - 18 n^{5} + s \left(9 n^{3}  + 12 d n^{2} -12 d^{2} n\right) 
  - s^{2} \left(32 d - 16 n\right) \right)}{n^{4} \left( 3 n-2d\right)^{2} \left(2 d + n\right)^{2}},$$
Let $v(s):=-48 d^{3} n^{2} + 72 d^{2} n^{3} + 12 d n^{4} - 18 n^{5} + s \left(9 n^{3}  + 12 d n^{2} -12 d^{2} n\right) 
  - s^{2} \left(32 d - 16 n\right)$.
Since $- \left(32 d - 16 n\right)<0$, $v$ is concave and it follows $v\ge \min\{v(0),v(0.6n^2) \} \ge 0$, since
\begin{align*}
&v(0)=
6 n^{2} \left(3 n -2d\right) \left(2 d - n\right) \left(2 d + n\right)> 0 \text{ and}\\
&v(0.6n^2)=\frac{3}{25} n^{2} \left(10 d - 3 n\right) \left(-40 d^{2} + 42 d n + 19 n^{2}\right)>0.
\end{align*}
\end{proof}

Now replacing $z$ in $f(N_1,L_2(N_1))$ with $z_{1}$ yields\\

%{\tiny
%\begin{align*}
\hspace{-3mm}$g:=\frac{- 16 d^{5} n^{2} + 32 d^{4} n^{3} + 16 d^{4} n s + 8 d^{3} n^{4} - 24 d^{3} n^{2} s - 4 d^{3} s^{2} - 24 d^{2} n^{5} - 4 d^{2} n^{3} s + 20 d^{2} n s^{2} - 9 d n^{6} + 6 d n^{4} s - 13 d n^{2} s^{2} + 32 d s^{3} - 12 n^{3} s^{2} - 16 n s^{3}}{\left(3 n-2d\right) \left(2 d + n\right) \left(- 2 d n - n^{2} + s\right) \left(- 2 d n + 3 n^{2} + s\right)},$\\
%\end{align*}
%}

where $f(N_1,L_2(N_1))\ge g$.

\begin{lemma}
    $$g\ge g_{1}:=d + \frac{4 s^{2}}{n \left(3n-2d\right)\left(2d+n\right)} - \frac{24 s^{3} \left(2 d - n\right)}{n^{2} \left(3 n-2d\right)^{2} \left(2 d + n\right)^{2}}$$
\end{lemma}

\begin{proof}
$$g-g_{1}=\frac{4 s^{4} \left(44 d^{2} n - 44 d n^{2} + 15 n^{3}- s(12 d  - 6 n)\right)}{n^{2} \left(3 n-2d\right)^{2} \left(2 d + n\right)^{2} \left(2 d n + n^{2} - s\right) \left(- 2 d n + 3 n^{2} + s\right)}>0,
$$
since $44 d^{2} n - 44 d n^{2} + 15 n^{3}- s(12 d  - 6 n)\ge \left(
44 d^{2} n - 44 d n^{2} + 15 n^{3}- s(12 d  - 6 n)\right)\big|_{s=0.5n^2}\\=2n(22d^2-25dn+9n^2)>0$.
\end{proof}

Hence $\lambda \ge f(N_1,L_2(N_1)) \ge g \ge g_1$ for $s\ge t$ and $\frac{n}{2}<d \le 0.8n$, which completes the proof of Theorem \ref{theoremF2}. \\

Now we show the remaining part of Theorem \ref{niki}, that is $s\ge s_0$ and $n/2 <d< 0.8n$ (see Lemma \ref{lemmaN}), where 
$$s_0= \frac{1}{2}\sqrt{\frac{n^{5} \left(2 \sqrt{d n}-\sqrt{2} n\right)}{\sqrt{d n} \left(d + n\right)}}.$$ We start with an upper bound for $r:=\frac{\sqrt{2} s^{2}}{2 n^{2} \sqrt{d n}}$.

\begin{lemma}
$$r_{1}:=\frac{s^{2}}{n^{3}} - \frac{s^{2} \left(d - \frac{n}{2}\right)}{n^{4}} + \frac{3 s^{2} \left(d - \frac{n}{2}\right)^{2}}{2 n^{5}}=\frac{s^{2} \left(12 d^{2} - 20 d n + 15 n^{2}\right)}{8 n^{5}}
 \ge r=\frac{\sqrt{2} s^{2}}{2 n^{2} \sqrt{d n}}$$

\end{lemma}

\begin{proof}
$$r_1^2-r^2=\frac{s^{4} \left(2 d - n\right)^{3} \left(18 d^{2} - 33 d n + 32 n^{2}\right)}{64 d n^{10}}>0
$$
\end{proof}

The bound in Theorem \ref{theoremF2} holds for $s\ge t=\frac{7 n \left(d - \frac{n}{2}\right)}{10}
$ and $n/2<d<0.8n$. 

\begin{lemma}
$t\le s_0$ for $n/2<d\le 0.8n$.
\end{lemma}

\begin{proof}
Let $\overline{s}:=\frac{1}{4}\sqrt{\frac{n^{3} \left(2 d - n\right) \left(7 n-6d\right)}{d + n}}$.
Since, as a function in $s$, $f_3-d-r=d + \frac{s^{2}}{n^{3}} -\frac{ \left(2 (d + n)\right) s^4}{n^{8}}-d-r$ has the zeros $0,\pm s_0$ and $f_3-d-r_{1}$ has the zeros $0,\pm \overline{s}$, the inequality
$f_3-d-r_{1}\le   f_3-d-r$ implies $\overline{s}\le s_0$. Now
$$\overline{s}^2-{\sa}^2=\frac{n^{2} \left(2 d - n\right) \left(-98 d^{2} - 199 d n + 224 n^{2}\right)}{400 \left(d + n\right)}>0,$$
since $-98 d^{2} - 199 d n + 224 n^{2}\ge \left( -98 d^{2} - 199 d n + 224 n^{2} \right)\big|_{d=0.8n}>2n^2$.
\end{proof}
Now the following lemma completes the proof of Theorem \ref{niki}.

\begin{lemma}
    $g_{1}\ge d+r_1$
\end{lemma}

\begin{proof}
$$ g_{1}-d-r_{1}=   \frac{s^{2} \left(2 d - n\right) \left(-96 d^{5} + 304 d^{4} n - 240 d^{3} n^{2} - 104 d^{2} n^{3} + 130 d n^{4} + 39 n^{5} - 192 n^{3} s\right)}{8 n^{5} \left(3 n-2d\right)^{2} \left(2 d + n\right)^{2}},$$
which, as a function in $s$, has the zeros $0$ and $s_1:=\frac{\left(3 n-2d\right) \left(2 d + n\right) \left(24 d^{3} - 52 d^{2} n + 26 d n^{2} + 13 n^{3}\right)}{192 n^{3}}
$.
Note that
$(\ref{havi2})$ implies $s\le h_1:=\left( n-d\right) \left(2 n - \sqrt{4 n \left( n-d\right)}\right)$. 
It is a mathematical standard task to show that $h_1 <s_1$ for all $\frac{n}{2}\le d \le 0.8n$, see Figure \ref{fig2}. Now $g_1-d-r_1 \ge 0$ follows easily.

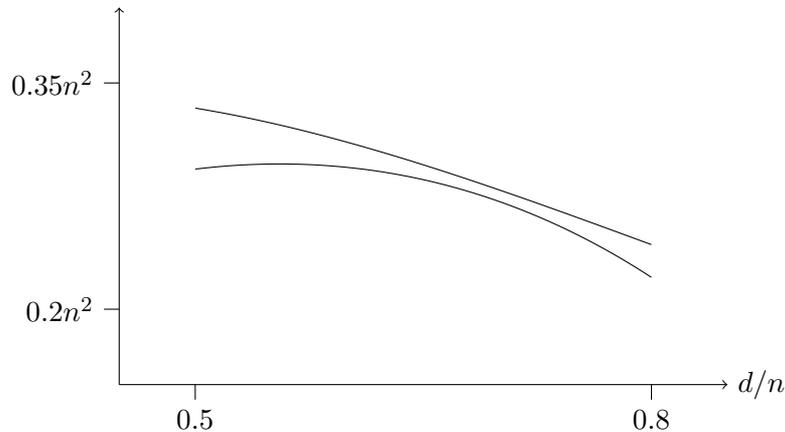
\begin{figure}[H]
\centering
\begin{tikzpicture}[yscale=20,xscale=20]
  \draw[->] (0.45,0.15) -- (0.85, 0.15) node[right] {$d/n$};
  \draw[-] (0.8,0.15) -- (0.8, 0.14) node[below] {$0.8$};
  \draw[->] (0.45, 0.15) -- (0.45,0.4) node[above] {};
  \draw[-] (0.45,0.35) -- (0.44,0.35) node[left] {$0.35n^2$};
  \draw[-] (0.45,0.2) -- (0.44,0.2) node[left] {$0.2n^2$};
  \draw[-] (0.5,0.15) -- (0.5, 0.14) node[below] {$0.5$};
  \draw[domain=0.5001:0.8, samples=50,variable=\x, black] plot ({\x}, {(3-2*\x)*(2*\x+1)*(24*\x*\x*\x-52*\x*\x+26*\x+13)/(192)});
  \draw[domain=0.5001:0.8, samples=50,variable=\x, black] plot ({\x}, {(1-\x)*(2-sqrt(4*(1-\x)))});
 % \draw[domain=0.5:0.8, smooth, variable=\x, black] plot ({\x},1);
\end{tikzpicture}
\caption{$h_1$ below $s_1$ }
\label{fig2}
\end{figure}
%
%
%For example, both $s_2$
%
%Comment: How to show $h_1< s_1$?\\
%$\left( n-d - 1\right) \left(2 n - \sqrt{4 n \left( n-d - 1\right) + 1} - 1\right)\le h_1:=\left( n-d\right) \left(2 n - \sqrt{4 n \left( n-d\right)}\right)$.
%Note by considering the second derivative, both $s_1$ and $h_1$ are concave for $\frac{n}{2}\le d \le 0.8n$. We consider the tangent $t$ of $s_1$ at the point $d=0.68n$. Now $s_1\le t$ and since $t|_{d=n/2}< s_1|_{d=n/2}$ and $t|_{d=0.8n}< s_1|_{d=0.8n}$, the result follows. \\
%
%Since $g_{app}-d-r_{app}$ has no zero for $\frac{n}{2}\le d\le 0.8n$ and $(g_{app}-d-r_{app})|_{d=\frac{2}{3}n}=\frac{s^{2} \left(1505 n^{2} - 5184 s\right)}{9800 n^{5}}
%>0$, the result follows.
\end{proof}

\end{document}